\documentclass[abstracton]{scrartcl} 

\usepackage{amsfonts}
\usepackage{amsmath}
\usepackage{amssymb}
\usepackage{xcolor}
\usepackage{dsfont}
\usepackage{epsf}
\usepackage{epstopdf}
\usepackage{graphics}
\usepackage{graphicx}
\usepackage{latexsym}
\usepackage{oldgerm}
\usepackage{dsfont}
\usepackage{amscd, amsthm,enumerate,verbatim,calc}
\usepackage{authblk}

\newcommand{\Aut}{\ensuremath{\operatorname{Aut}}}

\newtheorem{theorem}{Theorem}
\newtheorem{lemma}[theorem]{Lemma}

\newtheorem{problem}[theorem]{Problem}

\newcommand{\w}{\color{black}}
\newcommand{\s}{\color{black}}
\newcommand{\f}{\color{black}}

\usepackage[colorlinks,breaklinks,backref]{hyperref}
\usepackage{hyperref}
\usepackage{backref}

\begin{document}

\title
{Asymmetrizing infinite trees\\
\bigskip
}

\author[1]{Wilfried Imrich}
\author[2]{Rafa{\l}  Kalinowski}
\author[3]{Florian Lehner}
\author[2]{Monika Pil\'sniak}
\author[2]{Marcin Stawiski}
\affil[1]{\normalsize Montanuniversit\"at Leoben, A-8700 Leoben, Austria}
\affil[2]{\normalsize AGH University of Science and Technology, 30-059 Krakow, Poland}
\affil[3]{\normalsize University of Auckland, Auckland, New Zealand}

\date{\today}
\maketitle
\begin{abstract}
A graph $G$ is asymmetrizable if it has a set of vertices {\f whose setwise stablizer only consists of the identity automorphism. The}
motion $m$ of a graph  is the minimum number of vertices moved by any non-identity automorphism.  It is known that infinite trees $T$ with motion $m=\aleph_0$ are asymmetrizable if the vertex-degrees are bounded by $2^m.$ We show that this  also holds for arbitrary, infinite $m$, and that the number of inequivalent asymmetrizing sets is $2^{|T|}$.
\end{abstract}

\noindent
{\w Keywords: Automorphisms, infinite graphs, asymmetrizing and distinguishing number.\\               
Math. Subj. Class.: {05C05, 05C15, 05C25, 05C63.} }                      

\section{Introduction}\label{sec:intro}
{\f Given a graph $G$, a set $S \subseteq V(G)$ is called \emph{asymmetrizing} if the identity is only element of $\Aut G$ which setwise fixes $S$. A graph is called \emph{asymmetrizable}, if it has an asymmetrizing set.}

{\f Motivated by the use of asymmetric graphs in the construction of graphs with given automorphism group, asymmetrization was perhaps first studied in a 1977 paper by Babai \cite{ba-1977}, where he
proved that every tree $T$ in which all vertices have the same (finite or infinite) degree is asymmetrizable.}
{\f This result was later {\w reproved} and generalised by Polat and Sabidussi \cite{po-1991,posa-1991}.}

{\f Much recent work on asymmetrization was motivated by a popular paper by Albertson and Collins \cite{ac-1996}, and a lot of it has focussed on the connection between asymmetrization and}
the concept of \emph{motion} of a graph, which is {\f defined as} the minimum number of vertices moved by any non-identity automorphism.
This 
{\f connection} was already noted by Cameron, Neumann  and Saxl in \cite{canesa-1984}, where they 
{\f studied asymmetrizing sets for} permutation groups. For graphs the most notable recent result linking motion to asymmetrization  is due to Babai \cite{ba-2021}: he proved  that each connected, locally finite graph with infinite motion {\f is asymmetrizable}
thereby verifying the Infinite Motion Conjecture of Tucker \cite{tu-2011}.

The primary motivation for this paper is \cite[Question 4]{imtu-2020}, which asks whether each tree of motion $m>\aleph_0$ is asymmetrizable if its degrees are bounded by $2^m$.  
Our main theorem, Theorem \ref{thm:mainmain},  answers it affirmatively.
\begin{theorem}\label{thm:mainmain}
 Let $m$ be an infinite cardinal and $T$ be a  tree whose degrees are bounded by $2^m$. If the minimum number of vertices moved by each non-trivial automorphism of $T$ is  $m$,  then $T$ is asymmetrizable and the number of inequivalent asymmetrizing sets is $2^{|T|}$.
\end{theorem}

The bound $2^m$ is sharp, because the tree consisting of  more than $2^m$ rooted isomorphic asymmetric trees of order $m$ whose roots are connected to a common vertex  is not asymmetrizable.
Note that Theorem \ref{thm:mainmain} generalizes Babai's result from \cite{ba-1977},  because trees where all vertices have the  same infinite  degree $\alpha$  have motion $\alpha$ and thus satisfy the assumptions of the theorem. 
%
%

  In Section~\ref{sec:treelike} we apply Theorem \ref{thm:mainmain}
  to tree like-graphs. Regarding other applications we wish to   point out that the methods of this  paper can be used
to derive  results that are analogous to Theorem~\ref{thm:mainmain} for edge colorings.

Finally,
 it is worth mentioning that our methods only rely on ZFC and do not assume the Generalized Continuum Hypothesis, just as the  papers \cite{po-1991,posa-1991}, whose results we use.

\section{Preliminaries}\label{sec:prelim}
A graph $G$ is \emph{asymmetrizable} if it has {\f an \emph{asymmetrizing} set of vertices, that is,} a set $S \subseteq V(G)$ which is preserved only by the identity automorphism. If $S$ is such a set, then {\s its complement $V(G)\setminus S$} is also {\f asymmetrizing.}
The definition allows that $S$ or {\s $V(G)\setminus S$} are empty.
%
%

Two asymmetrizing sets $S$ of $G$ and $S'$ of $G'$ are called \emph{equivalent} if there exists an isomorphism $\varphi$ from $G$ to $G'$ such that $\varphi(S_1) = S_2$. 
Following \cite{posa-1991} we define the \emph{asymmetrizing number} of $G$, denoted by $a(T)$, as the number of pairwise inequivalent asymmetrizing sets. Observe that $a(G) \leq 2^{|V(G)|}$ for all graphs, and that $a(G) = 2^{|V(G)|}$ for asymmetric graphs.

Let $(T,w)$ denote the tree with root $w \in V(T)$, {\f and denote by $\Aut(T,w)$ the subgroup of $\Aut(T)$ that fixes $w$. By slight abuse of notation, we call a subset $S \subseteq V(T)$ \emph{asymmetrizing for $(T,w)$}, if the identity is the only element of $\Aut(T,w)$ which fixes $S$ setwise. Two asymmetrizing sets $S$ and $S'$ are called equivalent with respect to $\Aut(T,w)$, if there is an element of $\Aut(T,w)$ which maps $S$ to $S'$.}
We define 
$a(T,w)$ as the number of inequivalent asymmetrizing sets of $(T,w)$.




Let $(T,w)$ be a rooted {\f tree}. For vertices $x$, $y$ of $(T,w)$ we let $x \geq y$ denote the fact that $y$ lies on the unique path from the root to $x$. As usual, we say that $y$ is the parent of $x$ if $y<x$ and $xy$ is an edge. We say $x$ and $x'$ are siblings if they have the same parent, and  that $x$ and $x'$ are twins if they are siblings and if there is an automorphism which moves $x$ to $x'$ and fixes their parent. We call the set of twins of $x$ the \emph{similarity class} of $x$, denote it by $\bar{x}$, and set $\tau(x) =|\bar{x}|$. We always have $\tau(x) \geq 1$ since $x$ is a twin of itself.

For a vertex  $x$ of $(T,w)$ with parent $y$, we let $T^x$ denote the component of $T - xy$ which contains $x$. We consider $T^x$ as a rooted tree with root $x$ and write $a(T,w;x)$ for the number of inequivalent {\w asymmetrizing sets} of $T^x$. {\f If the rooted tree $(T,w)$ is clear from the context, we write $a(x)$ instead of $a(T,w;x)$; in particular, in this case we also write $a(w)$ instead of $a(T,w)$.}
Let $y$ be the parent of $x$ and $x'$.  Then clearly $x$ and $x'$ are twins if and only if $T^x$ and $T^{x'}$ {\s are isomorphic}.

Let $R_y$ be a set of representatives for the similarity classes of siblings of $y$. Then {\f by \cite[Theorem 2.3]{posa-1991}}

\begin{equation}\label{eq:1}
a(y) = \prod_{x\in R_w}{a(x) \choose \tau(x)},
\end{equation}
where ${a \choose \tau}$ denotes the usual binomial coefficient for finite $a$ and $\tau$. If $a$ is infinite, then ${a \choose \tau}$ is $a^\tau$  if $\tau\le a$, and 0 if $\tau > a$.

Equation (\ref{eq:1}) implies a helpful lemma that uses the concept of motion. {\w Recall from the introduction that} the {motion} $m(G)$ of a graph $G$  is the least number of vertices moved by a non-identity automorphism of $G$. For asymmetric graphs the motion is not defined, but we use the convention that $m(G) > \alpha$ for any asymmetric graph and any cardinal $\alpha$. In particular, a graph with fewer than $\alpha$ vertices has (by definition) motion $m(G)\geq \alpha$ if and only if it is asymmetric. This means that the order of a graph $G$ with motion $m$ is at least $m$ unless $G$ is asymmetric.

\begin{lemma}\label{le:1} Let $(T,w)$ be a rooted tree with motion $m\ge \aleph_0$, {\f all of} whose degrees {\s are bounded by $2^m$, and let} $y\in V(T)$. If $a(x) = 2^{|T^x|}$  for all {\s children $x$ of $y$}, then $a(y) = 2^{|T^y|}$.
\end{lemma}

\begin{proof} Let $(T,w)$ {\f and $y$} satisfy the assumptions of the lemma. 
Then for all siblings $x$ of $y$, {\s the subtree}   $T^x$ has motion $m$ and is asymmetrizable.

Let $x$ be a {\f child}
of $y$ 
{\f such that} $|T^x|<m$. Then $T^x$ is asymmetric, {\f and hence} $a(x) = 2^{|T^x|}$.
{\f Moreover} $\tau(x)=1$, {\f because otherwise the motion would be less than $m$}.

Now let {\f $x$ be a child of $y$ such that} $|T^x| \ge m$. Then $a(x) \ge 2^m \ge \tau(x)$ and hence
\[{a(x)\choose \tau(x)} =(2^{|T^x|})^{\tau(x)} = 2^{|T^x|\,\tau(x)}.\]
Because $|T^x| \,  \tau(x)$ is the total size of the union of the $V(T^z)$ for $z \in \bar{x},$ we conclude that
\[{a(x)\choose \tau(x)} = 2^{|{\bigcup_{z\in \bar{x}}V(T^z)}|}.\]

Substituting into Equation (\ref{eq:1}) we obtain
\begin{align*}
    a(y) &= \prod_{x\in R_y}{a(x) \choose \tau(x)}
    = \prod_{x\in R_w}2^{|{\bigcup_{z\in \bar{x}}V(T^z)|}}\\
    &= 2^{\sum_{x\in R_w}{|{\bigcup_{z\in \bar{x}}V(T^z)|}}}
    = 2^{|\bigcup_{x\in R_w}\bigcup_{z\in\bar{x}}V(T^z)|}\\
    &=2^{|T|}. \qedhere
\end{align*}
\end{proof}

\section{Proof of the main theorem}\label{sec:leaves}
{\f In this section we prove {\w Theorem} \ref{thm:mainmain}. The proof is split into three parts depending on the infinite paths that can be found in the tree. One-sided infinite paths are called \emph{rays} and  two-sided infinite paths  \emph{double rays}. We will discern three types of trees:
\begin{enumerate}
    \item rayless trees (also called compact trees) are treated in Theorem \ref{rayless},
    \item trees containing rays but no double rays (also called one-ended trees) are treated in Theorem \ref{noncomp}, and
    \item trees containing at least one double ray are treated in Theorem \ref{thm:main}.
\end{enumerate}}
{\s For convenience, we} will let $\Delta(T)$ denote the least upper bound on the degrees of the vertices in $T$. Note that if $\Delta(G)$ is infinite, then $\Delta(G)=|G|$ for every connected graph $G$.

\subsection{Compact trees}

Our proof {\f for compact trees} uses the concept of \emph{rank}, which was introduced by Schmidt in \cite{schm-1983}  
%
{\f and can be} inductively defined as follows.
\begin{itemize}
\item Finite trees have rank 0.
\item A tree $T$ has rank $\rho$ if
\begin{enumerate}
\item $T$ has not been assigned a rank less than $\rho$, and if
\item there is a finite set of $S$ vertices such that each component of $T - S$ has rank
less than $\rho$.
\end{enumerate}
\end{itemize}
In \cite{schm-1983} it was shown that {\f every rayless graph has a rank}, and that there is a tree
of rank $\rho$ for every ordinal number $\rho$. 
We will  need the following facts, shown in  \cite{schm-1983} and \cite{po-1994}.
{\f Firstly, }there is a canonical choice for the set $S$ in the definition above, by choosing $S$  minimally among all sets that work. This minimal set is called {\s the \emph{core}} of $T$ and can be shown to be unique; in particular, it is setwise fixed by
every automorphism of $T$. 
{\f Secondly, the rank cannot go up by removing additional vertices. In other words, if $S' \subseteq V(T)$ contains the core, then every component of $T -S'$ has rank less than $\rho$.}

Note that {\f the first fact above} implies that each rayless tree has a center consisting of {\f either} a single vertex or an edge that is preserved by all automorphisms. Just consider the minimal subtree $T_S$ of $T$ containing its core $S$. {\f This tree} $T_S$ is finite, {\f it is} preserved by all {\s automorphisms} of $T$, and thus {\f so is} its center. Despite the fact that this immediately follows from \cite{schm-1983} it was first explicitly stated in \cite{posa-1994}. {\f {\w The} second fact implies that when we remove all vertices of $T_S$ from $T$, then every component has strictly smaller rank than $T$.}

{\s Now we state and prove our main result for rayless trees.}

\begin{theorem} \label{rayless}
 Let $m$ be an infinite cardinal and {\f let} $T$ be a rayless tree with {\f motion} $m$ and
$\Delta(T)\leq 2^{m}$. Then $a(T) = 2^{|T|}$.
\end{theorem}
\begin{proof}
We use transfinite induction on the rank of $T$. Trees of rank 0 are finite, so they  have infinite motion only if they are asymmetric, and therefore $a(T)= 2^{|T|}$ {\f for every tree of rank 0 {\w that} satisfies the conditions of the theorem}.

For the induction step, let $\rho$  be any ordinal, {\f let} $T$ be a tree of rank $\rho$, and assume that
the statement of the theorem holds for any tree with rank $\sigma < \rho$. 
%
Let $T_S$ be the minimal subtree containing the core of $T$. {\f If $T$ has a central vertex, then let $w$ be this central vertex. Otherwise, let $w$ be one endpoint of the central edge. Consider the rooted tree $(T,w)$.}

{\f We claim that $a(x) = 2^{|T^x|}$ for every vertex $x$ of $(T,w)$. For vertices not contained in $T_S$ this is true by the induction hypothesis. Now assume that there is a vertex $x$ which does not satisfy the claim and let $y$ be such a vertex at maximal distance from $w$; note that the maximal distance is finite because $T_S$ is finite. All children of $y$ satisfy the claim, and by Lemma \ref{le:1} the claim is satisfied for $y$ as well.}

{\f If $T$ has a central vertex $w$, then the statement of the theorem follows immediately from the fact that $a(w) = a(T)$. If there is a central edge $ww'$, then there are $a(w)$ asymmetrizing sets of $(T,w)$ which {\w do} not contain $w'$ because the complement of an asymmetrizing set is again asymmetrizing. Clearly any such set is asymmetrizing for $T$ because its stabiliser must fix both $w$ and $w'$.}
%
%
%
\end{proof}

\subsection{One-ended trees}

{\f We now turn to the case of one-ended trees, that is, trees containing a ray, but no double ray.} We invoke a theorem of Polat \cite{po-1991} to prove the following theorem.

\begin{theorem} \label{noncomp}
 Let $T$ be a one-ended tree, and {\f let} $m$ be an infinite cardinal. If $m(T)= m$ and
$\Delta(T)\leq 2^{m}$, then $a(T)= 2^{|T|}$.
\end{theorem}
\begin{proof}
 Let $T$ satisfy the assumptions of the theorem. Then it contains a ray
 $R$. 
{\f For a vertex $x$ of $R$ we denote by $T^x$ the component of $T - E(R)$ which contains $x$. We consider $T^x$ as a rooted tree with root $x$ and set $a(x) = a(T^x, x)$.
Note that $T^x$ is necessarily rayless, and thus $a(x) = 2^{|T^x|}$ by Theorem \ref{rayless}.}
Combining this observation with \cite[Theorem 3.1]{po-1991}, we get
\[
    a(T,w_0) = \prod_{\w x\in V(R)}a(x) = \prod_{\w x\in V(R)}2^{|T^x|} = 2^{\sum_{\w x\in V(R)}{|T^x|}} = 2^{|T|}.
\]
If $w_0$ is {\f fixed by every element} {\color{black} of $\Aut (T)$}, then $a(T) = a(T, w_0)$. If not, then $a(T) = a(T, w_0)$ by \cite[Corollary 3.2]{po-1991}.
\end{proof}

\subsection{Trees with double rays}

{\f Finally, we consider the case where $T$ is} a tree containing double rays. {\f For such a tree~$T$}, we let $T_*$ be the tree induced by all vertices that lie on some double ray. For a vertex $x$ of $T_*$ we denote by $T^x$ the component of $T - E(T_*)$ which contains $x$. We consider $T^x$ as a rooted tree with root $x$ and set $a(x) = a(T^x, x)$.
{\f As above, note that $T^x$ is necessarily rayless, and thus $a(x) = 2^{|T^x|}$ by Theorem \ref{rayless}.}

Pick an arbitrary root $w$ of $T_*$ and define the concepts of parents, siblings, twins and $\tau(x)$ as in Section \ref{sec:intro}, {\f in particular recall that $\tau(x)$ is the number of twins of $x$}.
The following theorem is the equivalence between conditions (ii), (iv) and (v) of Theorem 3.5 in \cite{po-1991}.

\begin{theorem}
    \label{thm:polat}
    A tree $T$ which contains double rays is asymmetrizable if and only if $\tau(x) \leq \prod_{y \geq x} a (y)$ for every vertex $x$ of $T_*$. Moreover, in this case $a(T)=\prod_{x \in T_*} a (x)$.
\end{theorem}

We use it to prove {\f our main result for trees containing double rays}.

\begin{theorem}\label{thm:main}
   {\s Let $T$ be a tree} of infinite motion $m$, and {\f assume that} $\Delta(T)\le 2^m$. {\s If $T$ contains} a double ray, then $a(T)=2^{|T|}$.
\end{theorem}
\begin{proof}
    Pick an arbitrary root in $T_*$. By Theorem \ref{thm:polat} above, it suffices to show that $\tau(x) \leq \prod_{y \geq x} a (y)$ for every vertex $x$ of $T_*$. 
    Note that if $x$ and $x'$ are twins, then there is an automorphism of $T$ which swaps $x$ and $x'$ and only moves vertices in
    \[\bigcup_{y\geq x} T^y \cup \bigcup_{y\geq x'} T^y.\]
    If $|\bigcup_{y\geq x} T^y| < m$, then such an automorphism would move fewer than $m$ vertices, hence in this case $\tau(x) = 1$ which is less or equal than $\prod_{y \geq x} a (y),$ {\f {\w because}  all factors in the product are non-zero}.

    Hence we may assume that $|\bigcup_{y\geq x} T^y| \geq m$. Note that $\tau(x) \leq \Delta(T) \leq 2^m$.
    If there is some $y_0 \geq x$ such that $|T^{y_0}| \geq m$, then
    \[
        \tau(x) \leq 2^m \leq a(y_0) \leq \prod_{y \geq x} a (y).
    \]
    If there is no such $y$, then $a(y) = 2^{|T^y|}$ for every $y \geq x$, and thus
    \[
        \tau(x) \leq 2^m \leq 2^{|\bigcup_{y\geq x} T^y|} = 2^{\sum_{y\geq x} |T^y|} = \prod_{y \geq x} 2^{|T^y|} =  \prod_{y \geq x} a(y).
    \]
    {\f By the second part of} Theorem \ref{rayless} we conclude that
    \[
        a(T)=\prod_{x \in T_*} a (x) = \prod_{x \in T_*} 2^{|T^x|} = 2^{\sum _{x \in T_*} |T^x|} = 2^{|T|}. \qedhere
    \]
\end{proof}



\section{Tree-like graphs}
\label{sec:treelike}
We now apply our results to \emph{tree-like graphs}. In \cite{imkltr-07}, they  are defined as rooted graphs $(G,w)$, in which each vertex $y$ has a neighbor $x$ such that $y$ is on all shortest $x, w$-paths. By \cite[Theorem 4.2]{imkltr-07}
each tree-like graph $G$ with $\Delta\le 2^{\aleph_0}$ is asymmetrizable (as an unrooted graph). We present a proof of a   strengthened version  of this result.

In the terminology of the present paper we could also have defined tree-like graphs as rooted graphs $(G,w)$ in which each vertex has a child of which it is the only parent.
Note  that $w$ is the only vertex of $(G,w)$ that may have degree 1.

\begin{lemma}
Let $T$ be a  tree of infinite motion $m$ and  $\Delta(T)\le 2^{m}$. If each vertex of $T$ is on a double ray, then  $T$  has $2^{|T|}$ asymmetrizing sets $S$ in which each vertex of $S$ is adjacent to a vertex of $V(T)\setminus S$,  and to any $w\in V(T)$ there are $2^{|T|}$ asymmetrizing sets $S_w$ where $w$ is the only vertex of $S_w$ with no neighbor in $V(T)\setminus S_w$.
\end{lemma}

\begin{proof}
Form a new tree $T'$ from $T$ by choosing an arbitrary vertex $w\in V(T)$ and by subsequently contracting all edges $ab$ to single vertices if $d_T(w,a)$ is even and $d_T(w,b) = d_T(w,a) +1$.

  $T'$ has motion $m$, $\Delta(T')\le 2^m$, and $|T'| = |T|$. By Theorem \ref{thm:mainmain} $a(T') = 2^{|T'|} = 2^{|T|}$. From $a(T',w)\ge a(T')$ and $a(T,w) \ge a(T)$ we also infer that $a(T',w) = a(T,w) = 2^{|T|}$.

Let $\alpha\in \Aut(T,w)$ and $\alpha'$  its restriction to $V(T')$. Then $\alpha'\in \Aut(T',w)$, and $\alpha$ is uniquely determined by its action on $(T',w)$, because each vertex of $(T',w)$ different from $w$ has only one parent. This implies that a set $S\subseteq V(T)$ asymmetrizes $(T,w)$ if  $S' = S\cap V(T')$ asymmetrizes $(T',w)$.

Given an asymmetrizing set $S'$ of $(T',w)$ we extend it in two ways to an asymmetrizing set $S$ of $(T,w)$. The first is to set
$S = S'$. Clearly this implies that each vertex of $S$ has a neighbor in $V(T)\setminus S$ and, because $a(T',w) = 2^{|T|}$, the number of these  set is $2^{|T|}$.

The second way is to form as set $S$ by adding all children of $w$ to $S'$,  unless there is a child $u$ of $w$ all of whose children are in $S'$, then we only add $u$. Note that in this case $w$ has a neighbor that is not in $S$, because $T$ has no vertices of degree 1. Clearly we obtain $2^{|T|}$ asymmetrizing sets in this way, and either $w$ or $u$ are the only vertices of $S$  all of whose neighbors are in $S$.
\end{proof}

\begin{theorem} \label{thm:aleph}
Let  $G$ be a tree-like graph  with $\Delta(G) \le 2^{\aleph_0}$. Then $a(G) =2^{|G|}$.
\end{theorem}
\begin{proof}
Let  $G$ be a tree-like graph $(G,w)$. Then the set of edges $yx$, where $y$ is on  all shortest $x,w$-paths, are the edges of a spanning subgraph, say $F$. Clearly $F$ is a forest with no finite components. Each component $U$ has motion $\aleph_0$, unless it is asymmetric. Because $\Delta(U)\le \Delta(G) \le 2^{\aleph_0}$ we infer by Theorem \ref{thm:mainmain} that $a(U) = 2^{|U|}\ge 2^{\aleph_0}$.

Note that, by the definition of $F$, any automorphism $\alpha$  of $G$ that fixes $w$ preserves the components of $F$. This is  the case when $w$ has degree 1.

We now asymmetrize the components of $F$ under the following restrictions. Let $U_w$ be the component of $F$ that contains $w$. If $w$ has degree 1 we admit any asymmetrizing set $S$ of $U_w$, otherwise we only admit asymmetrizing sets $S$ in which  there is only vertex that has no neighbors in $V(U_w)\setminus S$.
For all other components $U$ of $F$ we only admit asymmetrizing sets $S$ where each vertex has a neighbor in $V(U)\setminus S$.

Because $F$ has at most $2^{\aleph_0}$ components we can asymmetrize them pairwise inequivalently with admitted asymmetrizing sets $S_U$. That is, we asymmetrize by sets $S_U$ that obey the above restrictions, and if  $U,U'$ are isomorphic components of $F$, then there is no isomorphism from $U$ to $U'$ that maps $S_U$ into $S_{U'}$.

Let $S$ be the union of all $S_U$ in  such a selection. If $w$ has degree 1, then each $\alpha\in \Aut(G)$ fixes $w$. Otherwise each $\alpha\in \Aut(G)$ that preserves $S$ fixes $w$ because $w$ is the only vertex of $S$ that has no neighbors in $V(G)\setminus S$. Hence, in either case each $\alpha\in \Aut(G)$ that preserves $S$ also preserves $F$. As each $S_U$ asymmetrizes $U$, the set $S$ asymmetrizes $F$, and thus also $G$, because $F$ is a spanning subgraph of $G$.

Let $\alpha$ be the least ordinal of the same cardinality as the set of components of $F$.   By transfinite induction with respect to the well ordering of the components of $F$ induced by $\alpha$, it is easily seen that there are $2^{|G|}$  asymmetrizing sets $S$ of $F$.
\end{proof}

\begin{problem} Let $m$ be an uncountable cardinal and $G$ be a tree-like graph with $m(G)=m$ and $\Delta(G)\le 2^m$. Is $G$ asymmetrizable?
\end{problem}

\end{document}